\documentclass[12pt]{article}
\usepackage[left=2.5cm,right=2.5cm,top=2.5cm,bottom=2.5cm,a4paper]{geometry}

\usepackage[a4paper]{geometry}
\usepackage{graphicx}
\usepackage{microtype}
\usepackage{siunitx}
\usepackage{booktabs}
\usepackage{cleveref}
\usepackage{graphics}
\usepackage{graphicx}
\usepackage{epsfig}
\usepackage{amsmath,amsfonts,amssymb,amsthm}
\usepackage{listings}
\usepackage{paralist}
\usepackage{sectsty}
\usepackage{amssymb,amsmath}
\usepackage{amsfonts}
\usepackage{xypic}
\usepackage{longtable}
\usepackage{tabularx,ragged2e,booktabs,caption}
\usepackage{array}
\usepackage{comment}
\usepackage{arydshln}
\newtheorem{theorem}{Theorem}[section]
\newtheorem{definition}[theorem]{Definition}

\newtheorem{corollary}[theorem]{Corollary}
\newtheorem{remark}[theorem]{Remark}
\newtheorem{example}[theorem]{Example}
\newtheorem{proposition}[theorem]{Proposition}

\newtheorem{open problem}[theorem]{Open problem}

\newcommand{\N}{\mathbb{N}}
\newcommand{\Z}{\mathbb{Z}}

\title
 {Minimality of $5$-adic polynomial dynamical systems} 

\title{Minimality of $5$-adic polynomial dynamical systems}

\author{Donggyun Kim, Youngwoo Kwon and Kyunghwan Song$^{*}$\\
	Department of Mathematics, Korea University, Seoul 02841, Republic of Korea \\
Corresponding Author's mail: heroesof@korea.ac.kr}

\begin{document}
\maketitle

\begin{abstract}

We characterize the dynamical systems consisting of the set of 5-adic integers and polynomial maps which consist of only one minimal component.

\end{abstract}


\section{Introduction}\label{sec:intro}

A dynamical system $(X,T)$ of a set $X$ and a self-map $T:X \rightarrow X$ studies the behavior of points in $X$ as $T$ is applied repeatedly. A \emph{minimal component in $X$ containing $x$ under $T$}, which is a generalization of the  orbit, is defined by
\[ \min\nolimits_{T}(x)=\{\, y\in X \mid T^i(y)=T^j(x) \ \text{ for some }\ i, j \in \Z_{\geq 0} \,\}. \]
Obviously the collection of all minimal components in $X$ forms a partition of $X$.

A goal of the study of the dynamical system is to decompose the set $X$ into minimal components for a given map $T$, which is called a \emph{minimal decomposition}. When $X$ is the set of $p$-adic integers, denote $\Z_p$, for a prime $p$, the  minimal decompositions of the dynamical systems for various maps $T$ are studied  in \cite{Coelho,Fan,FL}.

The number of components of the collection of the minimal components for the dynamical system $(\Z_p, T)$ is usually countably infinite \cite{Fan}. Hence it is an interesting case that the number of minimal components is only one, i.e., the set $\Z_p$ is a minimal component itself for a given map. We say the dynamical system is \emph{minimal} in this case.

There is a complete characterization of the minimal dynamical system on $\Z_p$ for polynomial maps when $p=2$ in \cite{L} and \cite{J}.  The latter paper used the ergodic terminologies instead of minimality. Complete results also exist for affine maps in \cite{Coelho,FLY} and for the square map in \cite{FsL} for all prime $p$. F. Durand and F. Paccaut \cite{DP} proved a complete characterization of the minimal dynamical system for polynomial maps on $\Z_p$ when $p=3$ and hoped that the method of this proof will also apply for $p\ge 5$.

In this paper, we provide a complete characterization of minimality of polynomial dynamical systems on the set of $p$-adic integers for $p=5$, using an extended method of the proof in \cite{DP}.



The rest of the paper is organized as follows. In Section 2, we give basic definitions of $p$-adic spaces and their dynamical systems. In Section 3, we state general properties of minimality for polynomial dynamical systems on $\Z_{p}$ and the results of minimality for polynomial dynamical systems on $\Z_{p}$ when $p=2$ and 3 in \cite{L} and \cite{DP}. In Section 4, we characterize the minimality for polynomial dynamical systems on $\Z_{p}$ when $p=5$.

\section{Preliminaries}\label{sec:pre}

We introduce basic theory of $p$-adic spaces and their dynamical systems, see for example \cite{R,BS}.

\subsection{$p$-adic spaces}

Let $\N$ be the set of positive integers, $\Z$ the set of integers and $p$ a prime number. We endow the ring $\Z/{p^n \Z}$ with the discrete topology for $n\in \N$ and the product $\prod_{n \in \N}\Z/{p^n \Z}$ with the product topology. Let $\varphi_n : \Z/p^{n+1} \Z ~\longrightarrow ~ \Z/p^n \Z$ be the canonical projection. Then we define the projective limit of $(\Z/p^{n}\Z, \varphi_{n})$, denoted by $\Z_{p}$, and call it the set of $p$-adic integers. This set can be seen as $\Z_p = \left\{ (x_n )_n \in \prod_{n\in \N} \Z/p^n \Z ~:~\varphi_n(x_{n+1})=x_n \right\}$,
which is a subset of $\prod_{n\in\N}\Z/p^n \Z$.
There are canonical projections $\pi_n : \Z_p ~\longrightarrow ~ \Z/p^n \Z$ and then $\varphi_n \circ \pi_{n+1} = \pi_n$.

A point $x_n$ can be represented by $x_n=a_0 + a_1 p+a_2 p^2 + a_3 p^3 + \cdots+a_{n-1}p^{n-1}$ with $a_i \in\{0,1,\cdots,p-1\}$ and the point $x = (x_n)$ in $\Z_{p}$ can be represented by $x=a_0 + a_1 p+a_2 p^2 + a_3 p^3 + \cdots$. Moreover, $\pi_n (x) = x_n$.

The $p$-adic order or $p$-adic valuation for $\Z_p$ is defined as $\nu_p : \Z_p \rightarrow \mathbb{Z}$
$$
\nu_p (x) = \begin{cases}\textrm{max}\{\nu \in \N : p^\nu | x \}&\textrm{if}~x \neq 0,\\
\infty &\textrm{if}~x =0.\end{cases}
$$

Then $p$-adic metric of $\Z_p$ is defined as $|\cdot |_p : \Z_p \rightarrow \mathbb{R}$
$$
|x|_p = \begin{cases}p^{-\nu_p (x)} &\textrm{if}~x \neq 0,\\
0 &\textrm{if}~x =0.\end{cases}
$$
With the $p$-adic metric, $\Z_p$ becomes a compact metric space. Note that $\Z/p^{n}\Z$ can be viewed as a subset of $\Z_{p}$,
$$\Z/p^{n}\Z=\{x\in\Z_{p}: a_{i}=0 \text{ for every } i\ge n\},$$
and so $\Z/p^n\Z$ is isomorphic to $\Z_p / p^n \Z_p$.

\subsection{Dynamical systems}
\begin{definition}
Given a set $X$, a \emph{$\sigma$-algebra} in $X$ is a collection $\mathcal{F}$ of subsets of $X$ such that;
\begin{enumerate}
\item[(1)] $\emptyset\in\mathcal{F}$.
\item[(2)] If $A\in\mathcal{F}$, then the complement of A is in $\mathcal{F}$, i.e., $A^c = X-A\in\mathcal{F}$.
\item[(3)] If $A_1 , A_2 , A_3,\cdots \in\mathcal{F}$, then $\bigcup A_i \in \mathcal{F}$.
\end{enumerate}
For a topological space $X$, the \emph{Borel $\sigma$-algebra} of $X$ is the $\sigma$-algebra generated by the open sets of $X$.
\end{definition}

\begin{definition}
Let $X$ be a compact metric space and $T : X \rightarrow X$  a homeomorphism. Let $\mathcal{B}$ be the Borel $\sigma$-algebra of $X$ and $\mu$ a probability measure defined on $\mathcal{B}$.
\begin{enumerate}
\item[(1)] $(X,T)$ is called a \emph{topological dynamical system}.
\item[(2)] A set $A\in\mathcal{B}$ is called \emph{invariant by} $T$ if $T^{-1}A=A$.
\item[(3)] The measure $\mu$ is called \emph{invariant by} $T$ if for every set $A\in\mathcal{B}$, $\mu(T^{-1}A)=\mu(A)$.
\item[(4)] The dynamical system $(X,T)$ is called \emph{ergodic for} $\mu$ if for every invariant set $A\in\mathcal{B}$, $\mu(A)=0$ or $\mu(A^c)=0$.

\item[(5)] The dynamical system $(X,T)$ is called \emph{minimal} if $A=\phi$ or $A=X$ whenever $A$ is a closed invariant set, or equivalently, every orbit $\textrm{Orb}_{T}(x)=\{T^{n}(x)~|~n\in\Z\}$ is dense in $X$, i.e., $\overline{\textrm{Orb}_T(x)}=X$. It follows that if $(X,T)$ is minimal, then $(X,T)$ is ergodic.


\item[(6)] Two topological dynamical systems $(X,T)$ and $(Y,S)$ are called \emph{conjugate} if there exists a homeomorphism $\psi: X\rightarrow Y$ such that $\psi\circ T=S\circ\psi$.
\end{enumerate}
\end{definition}

\begin{proposition}
\label{lem:mini}\cite{DP}
Let $n\in\N$ and $g:\Z/n\Z \rightarrow \Z/n\Z$ be a function. Then $g$ is minimal if and only if:
$$\begin{cases}
g^n (x)=x,\, \textrm{for every } x\in\Z/n\Z, \text{ and}\\
g^k (x)=x,\, \textrm{for every } x\in\Z/n\Z ~\Rightarrow ~k\equiv 0 ~({\rm mod}~n).
\end{cases}$$
\end{proposition}

\begin{definition}
A \textit{cycle} for a map $h:\Z/p^n\Z\rightarrow \Z/p^n\Z$ is a finite sequence $(x_k )_{0\le k\le s-1}$ such that for every $i\in\{0,\cdots, s-2\}$, $h(x_i)=x_{i+1}$ and $h(x_{s-1})=x_0$. A cycle for $h$ is called a \textit{full-cycle} if the number of elements of $\{x_k ~|~0\le k \le s-1\}$ is $p^n$.
\end{definition}

\section{Minimality for p-adic polynomial maps}\label{sec:poly}
\subsection{Properties of minimality for p-adic polynomial maps}

Let $f$ be a polynomial of coefficients in $\Z_p$. We view the polynomial as a polynomial map $f:\Z_p \rightarrow \Z_p$. It is a simple fact that the polynomial map is an $1$-Lipschitz map. Denote by $f_{/n}$ the induced map of $f$ on $\Z/p^n\Z_p$, i.e.,
\[ f_{/n}( x\!\! \mod p^n)=f(x) \quad \mod p^n. \]
Many properties of the dynamics $f$ are linked to those of $f_{/n}$.

\begin{theorem}\label{thm:minimal}\cite{DP}
Let $f:\Z_p \rightarrow \Z_p$ be a polynomial map over $\Z_p$ and $f_{/n}:\Z/p^n\Z \rightarrow \Z/p^n\Z$ be the induced map of $f$. Then the following are equivalent:
\begin{enumerate}
	\item[(1)] $f$ is minimal in $\Z_p$,
	\item[(2)] $f_{/n}$ is minimal in $\Z/p^n\Z$ for all  $n\in \N$,
	\item[(3)] $f_{/n}$ has a full-cycle in $\Z/p^n\Z$ for all  $n\in \N$.
\end{enumerate}
\end{theorem}

\begin{proposition}
\label{lem:minimal}\cite{DP}
Let $f:\Z_p \rightarrow \Z_p$ be a polynomial map over $\Z_p$ and $n\ge 1$. Suppose that $(\Z/p^n\Z,f_{/n})$ is minimal. The following are equivalent:
\begin{enumerate}
\item[(1)] $(\Z/p^{n+1}\Z,f_{/n+1})$ is minimal.
\item[(2)] For all $x\in\Z_p$, we have that $f^{p^n}(x)-x\notin p^{n+1}\Z_p$ and $(f^{p^n})'(x)\in 1+p\Z_p$.
\item[(3)] There exists $x\in\Z_p$ such that $f^{p^n}(x)-x\notin p^{n+1}\Z_p$ and $(f^{p^n})'(x)\in 1+p\Z_p$.
\end{enumerate}
\end{proposition}

\begin{proposition}\label{prop:minimal}\cite{DP}
Let $f:\Z_p\rightarrow \Z_p$ be a polynomial. Then $(\Z_p,f)$ is minimal if and only if $(\Z/p^{\delta}\Z,f_{/\delta})$ is minimal, where $\delta=3 \textrm{ if } p \textrm{ is } 2 \textrm{ or } 3 \textrm{ and } \delta=2 \textrm{ if } p\geq 5.$
\end{proposition}

\begin{remark}\label{remark:minimal}
Let $f(x)=a_0 + a_1 x + a_2 x^2 + \cdots + a_d x^d$ be a polynomial over $\Z_p$, hence $f(0)=a_0$. If $f$ is minimal, then $a_0\not\equiv 0~(mod~p)$; otherwise, $0$ is a fixed point modulo $p$ for $f$. Define a polynomial map $\psi : \Z_p \rightarrow \Z_p$ by $\psi (x) = (1/a_0) x$, and $g(x) = (1/a_0)f(a_0 x)= 1+a_1 x + a_2 a_0 x^2 + a_3 a_0^2 x^3 + \cdots + a_d a_0^{d-1} x^d$. Then $(\Z_p,f)$ and $(\Z_p,g)$ are conjugate since $\psi \circ f = g \circ \psi$.
Therefore, $f$ is minimal if and only if $g$ is minimal and moreover $g(0)=1$.  Hence we will restrict to polynomials $f$ with $f(0)=1$ without loss of generality.
\end{remark}

\subsection{Minimality for p-adic polynomial maps for $p=2$ and $3$}

We review the minimality of a polynomial $f$ over $\Z_p$ in the dynamical system $(\Z_{p},f)$ which is treated for the cases of $p=2$ and $p=3$ in Durand and Paccaut's paper  \cite{DP}.

Let $f(x)=a_d x^d + a_{d-1}x^{d-1}+\cdots+a_1 x+1$ be a polynomial over $\Z_p$ for $p=2$ or $3$.  We define the following terms derived from $f$ which are used to explain the minimality conditions of $f$.
\begin{align*}
	& A_1 =\sum_{i\in 1+2\N}a_i,
    & A_2 =\sum_{i\in 2\N}a_i,&
	& D_1=\sum_{i\in \N}ia_i, &
    & D_{-1}=\sum_{i\in\N} ia_i (-1)^{i-1}.
\end{align*}

From the derivative of $f(x)$, $f'(x)=d a_d x^{d-1} + (d-1)a_{d-1}x^{d-2}+\cdots+2 a_2 x+ a_1$, we can check that $D_1=f'(1)$ and $D_{-1}=f'(-1)$.

\begin{theorem}[\cite{DP}]\label{thm:2minimal}
The dynamical system $(\Z_2,f)$ is minimal if and only if
\begin{eqnarray*}
a_1 &\equiv&1~({\rm mod}~2),\\
A_1 &\equiv& 1 ~({\rm mod}~2),\\
A_1+A_2&\equiv&1~({\rm mod}~4),\\
2a_2+a_1A_1&\equiv& 1~({\rm mod}~4).
\end{eqnarray*}
\end{theorem}

\begin{theorem}[\cite{DP}]\label{thm:3minimal}
The dynamical system $(\Z_3,f)$ is minimal if and only if $f$ fulfils one of the following four cases:
\begin{center}
\begin{tabular}
{|m{1.2cm} | m{0.4cm} m{0.4cm} m{0.4cm} m{0.4cm} m{0.05cm} m{0.5cm}:m{5.0cm}|}
\hline
& $A_1$ & $A_2$ & $D_1$ & $D_{-1}$ & $a_1$ & $(3)$ & \\ \hline
case I & $1$ & $0$ & $2$ & $2$ & $1$ & & $A_1+5\not\equiv0 ~ (9)$, \par $A_1+5\not\equiv 3a_2 +3\sum_{j\ge0} a_{5+6j}~ (9)$\\ \hline
case II & $1$ & $0$ & $1$ & $1$ & $1$ & & $A_2+6\not\equiv0 ~ (9)$, \par $A_0+6\not\equiv 6a_2 +3\sum_{j\ge0} a_{2+6j}~ (9)$\\ \hline
case III & $1$ & $0$ & $1$ & $2$ & $2$ & & $A_1+5\not\equiv0 ~ (9)$, \par $A_1+5\not\equiv 6a_2 +3\sum_{j\ge0} a_{5+6j}~ (9)$ \\ \hline
case IV & $1$ & $0$ & $2$ & $1$ & $2$ & & $A_2+6\not\equiv0 ~ (9)$, \par $A_0+6\not\equiv 3a_2 +3\sum_{j\ge0} a_{2+6j}~ (9)$
 \\ \hline
\end{tabular}
\captionof{table}{Minimality conditions for $(\Z_3,f)$}
\end{center}
\end{theorem}


\section{Minimality for 5-adic polynomial maps}

The minimality of a polynomial $f$ over $\Z_p$ in the dynamical system $(\Z_{p},f)$ has treated for the cases of $p=2$ and $p=3$ in Durand and Paccaut's paper \cite{DP}. We extended the result to the case of $p=5$.

From Proposition \ref{prop:minimal}, for the case of $p \geq 5$, $(\Z_{p},f)$ is minimal if and only if $(\Z/p^2\Z, f_{/2})$ is minimal. Proposition \ref{lem:minimal} explains that when $(\Z/p\Z, f_{/1})$ is minimal, how we obtain the minimality condition of $(\Z/p^2\Z, f_{/2})$.
We combine these results as a corollary.

\begin{corollary}\label{cor:minimal_Z_5f}
  The dynamical system $(\Z_5, f)$ is minimal if and only if $f$ satisfies the following:
  \begin{enumerate}
    \item[(1)] $f_{/1}$ has a full cycle in $\Z/{5\Z}$,
	\item[(2)] $(f^5)'(0) \equiv 1 ~({\rm mod}~5)$ and
	\item[(3)] $(f^5)(0) \not\equiv 0 ~({\rm mod}~25)$.
  \end{enumerate}
\end{corollary}


Let $f(x)=a_d x^d + a_{d-1}x^{d-1}+\cdots+a_1 x+1$ be a polynomial over $\Z_5$ of degree $d$.  We define the following terms derived from $f$ which are used to explain the minimality conditions of $f$.

\begin{align*}
	& A_1=\sum_{i\in1+4\N}a_i, &
    & A_2 =\sum_{i\in2+4\N} a_i, &
    & A_3 =\sum_{i\in3+4\N}a_i, &
    & A_4 =\sum_{i\in 4\N}a_i, \\
	& D_1=\sum_{i\in \N}ia_i, &
	& D_{-1}=\sum_{i\in\N} ia_i (-1)^{i-1},&
    & D_2=\sum_{i\in \N} i a_i 2^{i-1}, &
    & D_{-2}=\sum_{i\in \N} i a_i (-2)^{i-1}.
\end{align*}

From the derivative of $f(x)$, $f'(x)=d a_d x^{d-1} + (d-1)a_{d-1}x^{d-2}+\cdots+2 a_2 x+ a_1$, we can check that $D_1=f'(1),\  D_{-1}=f'(-1),\ D_2=f'(2)$ and $D_{-2}=f'(-2)$.

\begin{proposition}\label{prop:minimal_cond1}
  The dynamical system $(\Z/5\Z,f_{/1})$ is minimal if and only if $f$ satisfies one of the following cases:
  \begin{center}
  \begin{tabular}{|l|c c c c c|}
    \hline
    & $A_1$ & $A_2$ & $A_3$ & $A_4$ & $({\rm{mod}}~5)$ \\ \hline
	$\textrm{case I}$ & $1$ & $0$ & $0$ & $0$ &\\
	$\textrm{case II}$ & $4$ & $4$ & $3$ & $0$ &\\
	$\textrm{case III}$ & $1$ & $3$ & $3$ & $0$ &\\
	$\textrm{case IV}$ & $1$ & $4$ & $2$ & $0$ &\\
	$\textrm{case V}$ & $4$ & $2$ & $2$ & $0$ &\\
	$\textrm{case VI}$ & $0$ & $0$ & $3$ & $0$ &\\ \hline
  \end{tabular}
  \captionof{table}{The minimality conditions for $(\Z/5\Z,f_{/1})$}\label{table:minimal_cond1}
  \end{center}
\end{proposition}

\begin{proof} Since $x^5 \equiv x ~({\rm mod}~5)$ for any $x$, we can write that $f_{/1}(x)\equiv A_4x^4 + A_3x^3 + A_2x^2 + A_1x + 1 ~ ({\rm{mod}}~5)$. Then, we have that:
\begin{eqnarray*}
f_{/1}(0)&\equiv& 1 ~({\rm mod}~5),  \\
f_{/1}(1)&\equiv& A_4 + A_3 + A_2 + A_1 + 1~({\rm mod}~5),\\
f_{/1}(2)&\equiv& A_4 - 2A_3 - A_2 + 2A_1 + 1~({\rm mod}~5),\\
f_{/1}(3)&\equiv& A_4 + 2A_3 - A_2 - 2A_1 + 1~({\rm mod}~5),\\
f_{/1}(4)&\equiv& A_4 - A_3 + A_2 - A_1 + 1~({\rm mod}~5).
\end{eqnarray*}

By Theorem \ref{thm:minimal}, $f_{/1}$ is minimal if and only if  $f_{/1}$ has a full-cycle in $\Z/5\Z$. In order that $f_{/1}$ has a full-cycle, since $f_{/1}(0)=1$ and $f_{/1}^5(0)=f_{/1}^4(1)=0$, we should match the values in $\{ f_{/1}(1),f_{/1}^2(1),f_{/1}^3(1) \}$ with in $\{2,3,4\}$, bijectively. There are six choices for this match. Therefore, $f_{/1}$ is minimal if and only if  $f_{/1}$ must satisfy one of the following six cases.

\begin{description}
\item[case] I: The polynomial $f$ has the full-cycle as $f_{/1}(1) = 2, f_{/1}(2) = 3, f_{/1}(3) = 4, f_{/1}(4) = 0$ and $f_{/1}(0) = 1$. In this case, we have the following
\begin{eqnarray*}
f_{/1}(1)&\equiv& A_4 + A_3 + A_2 + A_1 + 1\equiv2~({\rm mod}~5),\\
f_{/1}(2)&\equiv& A_4 - 2A_3 - A_2 + 2A_1 + 1\equiv3~({\rm mod}~5),\\
f_{/1}(3)&\equiv& A_4 + 2A_3 - A_2 - 2A_1 + 1\equiv4~({\rm mod}~5),\\
f_{/1}(4)&\equiv& A_4 - A_3 + A_2 - A_1 + 1\equiv0~({\rm mod}~5).
\end{eqnarray*}
By computing the system of congruences, we obtain that $A_1\equiv1 \text{ and } A_2\equiv A_3\equiv A_4\equiv0~({\rm mod}~5)$.

Similarly, we can find the values $A_{1},~A_{2},~A_{3}$ and $A_{4}$ for other cases.

\item[case] II: The polynomial $f$ has the full-cycle as $f_{/1}(1) = 2, f_{/1}(2) = 4, f_{/1}(4) = 3, f_{/1}(3) = 0$ and $f_{/1}(0) = 1$.
In this case, we obtain that $A_1\equiv A_2\equiv4, ~~A_3\equiv3 \text{ and } A_4\equiv0~({\rm mod}~5)$.

\item[case] III: The polynomial $f$ has the full-cycle as $f_{/1}(1) = 3, f_{/1}(3) = 2, f_{/1}(2) = 4, f_{/1}(4) = 0$ and $f_{/1}(0) = 1$.
In this case, we obtain that $A_1\equiv1,~~ A_2\equiv A_3\equiv3 \text{ and } A_4\equiv0~({\rm mod}~5)$.

\item[case] IV: The polynomial $f$ has the full-cycle as $f_{/1}(1) = 3, f_{/1}(3) = 4, f_{/1}(4) = 2, f_{/1}(2) = 0$ and $f_{/1}(0) = 1$.
In this case, we obtain that $A_1\equiv1,~~ A_2\equiv4,~~ A_3\equiv2 \text{ and } A_4\equiv0~({\rm mod}~5)$.

\item[case] V: The polynomial $f$ has the full-cycle as $f_{/1}(1) = 4, f_{/1}(4) = 2, f_{/1}(2) = 3, f_{/1}(3) = 0$ and $f_{/1}(0) = 1$.
In this case, we obtain that $A_1\equiv4,~~ A_2\equiv A_3\equiv2 \text{ and } A_4\equiv0~({\rm mod}~5)$.

\item[case] VI: The polynomial $f$ has the full-cycle as $f_{/1}(1) = 4, f_{/1}(4) = 3, f_{/1}(3) = 2, f_{/1}(2) = 0$ and $f_{/1}(0) = 1$.
In this case, we obtain that $A_1\equiv A_2\equiv0,~~ A_3\equiv3 \text{ and } A_4\equiv0~({\rm mod}~5)$.
\end{description}
Therefore, we obtain the statement.
\end{proof}

\begin{proposition}\label{prop:minimal_cond2}
$(f^5)'(0) \equiv 1 ~({\rm mod}~5)$ if and only if $a_1 D_1 D_2 D_{-2} D_{-1} \equiv 1 ~({\rm mod}~5)$.
\end{proposition}

\begin{proof}
We compute that:
\begin{eqnarray*}
	(f^5)'(0)&\equiv& \prod_{i=0}^{4} f'(f^i(0))\\
	&\equiv& f'(0)f'(1)f'(2)f'(3)f'(4)\\
	&\equiv& a_1 \sum_{i=1}^{d}i a_i \cdot\sum_{i=1}^{d}i a_i  2^{i-1} \cdot\sum_{i=1}^{d}i a_i (-2)^{i-1} \cdot\sum_{i=1}^{d} i a_i (-1)^{i-1}\\
	&\equiv& a_1 D_1 D_2 D_{-2} D_{-1} ~({\rm mod}~5).
\end{eqnarray*}
Therefore, we obtain the statement.
\end{proof}

\begin{theorem}\label{thm:minimal5}
The dynamical system $(\Z_5,f)$ is minimal if and only if the terms $A_1, A_2, A_3, A_4, D_1, D_{-1}, D_2, D_{-2}$ and $a_1$ derived from the polynomial $f$ satisfy one of the following six cases:

\begin{center}
  \begin{tabular}{|l|l|}	\hline
    $\textrm{case I}$ & $A_1 = 1 + \alpha_1 5,\ A_2 = 0 + \alpha_2 5,\ A_3 = 0 + \alpha_3 5,\ A_4 = 0 + \alpha_4 5$, \\
	& $a_1 D_1 D_2 D_{-2} D_{-1} \equiv 1 ~({\rm mod}~5)$, \\
	& $(4\alpha_1 + \alpha_2 + 4\alpha_3 + \alpha_4) $ \\
	& $ \quad + (3\alpha_1 + 4\alpha_2 + 2\alpha_3 + \alpha_4)D_{-1} +(1 + 2\alpha_1 + 4\alpha_2 + 3\alpha_3 + \alpha_4)D_{-1} D_{-2}$ \\
	& $\quad \quad + (\alpha_1 + \alpha_2 + \alpha_3 + \alpha_4)D_{-1} D_2 D_{-2} \not\equiv 0 ~({\rm mod}~5)$. \\
	\hline
	$\textrm{case II}$ & $A_1 = 4 + \alpha_1 5,\ A_2 = 4 + \alpha_2 5,\ A_3 = 3 + \alpha_3 5,\ A_4 = 0 + \alpha_4 5$, \\
	& $a_1 D_1 D_2 D_{-2} D_{-1} \equiv 1 ~({\rm mod}~5)$, \\
	& $(4 + 3\alpha_1 + 4\alpha_2 + 2\alpha_3 + \alpha_4) $\\
	& $\quad + (4\alpha_1 + \alpha_2 + 4\alpha_3 + \alpha_4)D_{-2} +(2\alpha_1 + 4\alpha_2 + 3\alpha_3 + \alpha_4)D_{-1} D_{-2}$ \\
	& $\quad \quad + (2 + \alpha_1 + \alpha_2 + \alpha_3 + \alpha_4) D_{-1} D_2 D_{-2} \not\equiv 0 ~({\rm mod}~5)$. \\
	\hline
	$\textrm{case III}$ & $A_1 = 1 + \alpha_1 5,\ A_2 = 3 + \alpha_2 5,\ A_3 = 3 + \alpha_3 5,\ A_4 = 0 + \alpha_4 5$, \\
	& $a_1 D_1 D_2 D_{-2} D_{-1} \equiv 1 ~({\rm mod}~5)$. \\
	\hline
	$\textrm{case IV}$ & $A_1 = 1 + \alpha_1 5,\ A_2 = 4 + \alpha_2 5,\ A_3 = 2 + \alpha_3 5,\ A_4 = 0 + \alpha_4 5$, \\
	& $a_1 D_1 D_2 D_{-2} D_{-1} \equiv 1 ~({\rm mod}~5)$, \\
	& $(4 + 2\alpha_1 + 4\alpha_2 + 3\alpha_3 + \alpha_4) $\\
	& $\quad  + (4\alpha_1 + \alpha_2 + 4\alpha_3 + \alpha_4) D_2 +(3\alpha_1 +  4\alpha_2 + 2\alpha_3 + \alpha_4 ) D_{-1} D_2 $ \\
	& $\quad \quad + (2 + \alpha_1 + \alpha_2  + \alpha_3 + \alpha_4) D_{-1} D_2 D_{-2} \not\equiv 0 ~({\rm mod}~5)$. \\
	\hline
	$\textrm{case V}$ & $A_1 = 4 + \alpha_1 5,\ A_2 = 2 + \alpha_2 5,\ A_3 = 2 + \alpha_3 5,\ A_4 = 0 + \alpha_4 5$, \\
	& $a_1 D_1 D_2 D_{-2} D_{-1} \equiv 1 ~({\rm mod}~5)$, \\
	& $(4 +  3\alpha_1 + 4\alpha_2 + 2\alpha_3 + \alpha_4) $ \\
	& $\quad + (1 + 2\alpha_1 + 4\alpha_2 + 3\alpha_3 + \alpha_4) D_{-2} +(4 + 4\alpha_1 + \alpha_2 + 4\alpha_3 + \alpha_4) D_2 D_{-2} $\\
	& $\quad \quad + (2 + \alpha_1 + \alpha_2 + \alpha_3 + \alpha_4) D_{-1} D_2 D_{-2}  \not\equiv 0 ~({\rm mod}~5)$. \\
	\hline
	$\textrm{case VI}$ & $A_1 = 0 + \alpha_1 5,\ A_2 = 0 + \alpha_2 5,\ A_3 = 3 + \alpha_3 5,\ A_4 = 0 + \alpha_4 5$, \\
	& $a_1 D_1 D_2 D_{-2} D_{-1} \equiv 1 ~({\rm mod}~5)$, \\
	& $(4 + 2\alpha_1 + 4\alpha_2 + 3\alpha_3 + \alpha_4) $ \\
	& $\quad + (1 + 3\alpha_1 + 4\alpha_2  + 2\alpha_3 + \alpha_4) D_2 +(4\alpha_1 + \alpha_2 + 4\alpha_3 + \alpha_4 )D_2 D_{-2} $ \\
	& $\quad \quad + (1 + \alpha_1 + \alpha_2 + \alpha_3 + \alpha_4) D_{-1} D_2 D_{-2} \not\equiv 0~({\rm mod}~5)$. \\
	\hline
  \end{tabular}
  \captionof{table}{Minimality conditions for $(\Z_5,f)$}
  \end{center}
Where, for $1\leq i \leq 4$, each $\alpha_i$ is an element of $\Z_5$.
\end{theorem}

\begin{proof}
    From Corollary \ref{cor:minimal_Z_5f}, $(\Z_5,f)$ is minimal if and only if $f$ satisfies the three conditions. the first condition which is that $f_{/1}$ has a full cycle in $\Z/5\Z$ is equivalent to that $f$ satisfies one of the six cases in Table \ref{table:minimal_cond1}, by Proposition \ref{prop:minimal_cond1}. The second condition which is that $(f^5)'(0) \equiv 1 ~ (\mod 5)$ is equivalent to that $a_1 D_1 D_2 D_{-2} D_{-1} \equiv 1~({\rm mod}~5)$ by Proposition \ref{prop:minimal_cond2}.

    Before we examine the third condition, we compute the following for convenience.
  \begin{enumerate}
	\item Let $T\equiv 1 ~({\rm mod}~5)$. We obtain that
		\begin{align*}
		f(T) & = 1+\sum_{i=1}^{d}a_{i} (1+(T-1))^{i}\\
		&\equiv  1+\sum_{i=1}^{d}a_{i}(1^{i} + i \cdot 1^{i-1} (T-1))~({\rm mod}~25)\\
		&\equiv  A_4 + A_3 + A_2 + A_1 + 1 +D_1 (T-1)~({\rm mod}~25).\\
		\intertext{Therefore,}
		f(T)&\equiv A_4 + A_3 + A_2 + A_1 + 1 ~({\rm mod}~5).
		\end{align*}

	\item Let $T\equiv 2 ~({\rm mod}~5)$.  We obtain that
		\begin{align*}
		f(T) & = 1+\sum_{i=1}^{d}a_{i} (2+(T-2))^{i}\\
		&\equiv  1+\sum_{i=1}^{d}a_{i}(2^{i} + i \cdot 2^{i-1} (T-2))~({\rm mod}~25)\\
		&\equiv  A_4 -2 A_3 - A_2 + 2A_1 + 1 + D_2 (T-2)~({\rm mod}~25). \\
		\intertext{Therefore,}
		f(T)&\equiv A_4 -2 A_3 - A_2 + 2A_1 + 1 ~({\rm mod}~5).
		\end{align*}

	\item Let $T\equiv -2 ~({\rm mod}~5)$.  We obtain that
		\begin{align*}
		f(T) & = 1+\sum_{i=1}^{d}a_{i} ((-2)+(T+2))^{i}\\
		&\equiv  1+\sum_{i=1}^{d}a_{i}((-2)^{i} + i \cdot (-2)^{i-1} (T+2))~({\rm mod}~25)\\
		&\equiv  A_4 +2 A_3 - A_2 - 2A_1 + 1 + D_{-2} (T+2)~({\rm mod}~25).\\
		\intertext{Therefore,}
		f(T)&\equiv A_4 +2 A_3 - A_2 - 2A_1 + 1  ~({\rm mod}~5).
		\end{align*}

	\item Let $T\equiv -1 ~({\rm mod}~5)$. We obtain that
		\begin{align*}
		f(T) & = 1+\sum_{i=1}^{d}a_{i} ((-1)+(T+1))^{i}\\
		&\equiv  1+\sum_{i=1}^{d}a_{i}((-1)^{i} + i \cdot (-1)^{i-1} (T+1))~({\rm mod}~25)\\
		&\equiv  A_4 - A_3 + A_2 - A_1 + 1 + D_{-1} (T+1)~({\rm mod}~25). \\
		\intertext{Therefore,}
		f(T)&\equiv A_4 - A_3 + A_2 - A_1 + 1 ~({\rm mod}~5).
		\end{align*}
	\end{enumerate}

 Now we will examine conditions for $f^5(0) \not\equiv 0 ~({\rm mod}~25)$ which is the third condition of Corollary \ref{cor:minimal_Z_5f}. In examining the condition, we will check the six cases in Proposition \ref{prop:minimal_cond1} one by one.

\begin{description}
	\item[case] I: Consider the case I in Proposition \ref{prop:minimal_cond1}, which is that
	\[ 	A_1\equiv 1, \text{ and } A_2\equiv A_3\equiv A_4\equiv 0~({\rm mod}~5). \]
	We have that $f(0)=1$ and let $T_2=f^2(0)$, then
	\[ T_2=f^2(0)=f(1)= A_4 + A_3 + A_2 + A_1 + 1 \]
	and $T_2 \equiv 2~({\rm mod}~5)$. Let $T_3=f^3(0)$, then
	\[ T_3\equiv f^3(0)\equiv f(T_2)\equiv A_4 - 2A_3 - A_2 + 2A_1 + 1 + D_2 (T_2-2)~~~({\rm mod}~25) \]
	and $T_3 \equiv -2~({\rm mod}~5)$. Let $T_4=f^4(0)$, then
	\[ T_4\equiv f^4(0)\equiv f(T_3)\equiv A_4 + 2A_3 - A_2 - 2A_1 + 1 + D_{-2} (T_3+2)~({\rm mod}~25) \]
	and $T_4 \equiv -1~({\rm mod}~5)$. Finally we compute $f^5(0)$:
	\[ f^{5}(0)\equiv f(T_4)\equiv A_4 - A_3 + A_2 - A_1 + 1 + D_{-1} (T_4+1)~({\rm mod}~25). \]
	Thus,
	\begin{eqnarray*}
		f^{5}(0)&\equiv& A_4 - A_3 + A_2 - A_1 + 1 + D_{-1} (T_4+1)\\
		&\equiv& A_4 - A_3 + A_2 - A_1 + 1 + D_{-1} (A_4 + 2A_3 - A_2 - 2A_1 + 1\\
		&&\quad + D_{-2} (T_3+2)+1)\\
		&\equiv& A_4 - A_3 + A_2 - A_1 + 1 + D_{-1} (A_4 + 2A_3 - A_2 - 2A_1 + 1\\
		&&\quad + D_{-2} (A_4 - 2A_3 - A_2 + 2A_1 + 1 + D_2 (T_2-2)+1))\\
		&\equiv& A_4 - A_3 + A_2 - A_1 + 1 + D_{-1} (A_4 + 2A_3 - A_2 - 2A_1 + 1\\
		&&\quad + D_{-2} (A_4 - 2A_3 - A_2 + 2A_1 + 1 \\
		&&\quad \quad + D_2 (A_4 + A_3 + A_2 + A_1 + 1 -2)+1))~({\rm mod}~25).
	\end{eqnarray*}
Substitute $A_1 \equiv 1 + \alpha_1 5, A_2 \equiv \alpha_2 5, A_3 \equiv \alpha_3 5$, and $A_4 \equiv \alpha_4 5$ with $\alpha_i \in \Z_5$ and then we obtain
\begin{align*}
	f^{5}(0)\equiv& 5\bigl((4\alpha_1 + \alpha_2 + 4\alpha_3 + \alpha_4) \\
	 & \quad + (3\alpha_1 + 4\alpha_2 + 2\alpha_3 + \alpha_4)D_{-1} + (1 + 2\alpha_1 + 4\alpha_2 + 3\alpha_3 + \alpha_4)D_{-1} D_{-2} \\
	& \quad\quad + (\alpha_1 + \alpha_2 + \alpha_3 + \alpha_4)D_{-1} D_2 D_{-2}\bigr) ~({\rm mod}~25).
\end{align*}
In order to have $f^5(0) \not\equiv 0 ~({\rm mod}~25)$, the terms derived from $f$ must satisfy
\begin{align*}
&(4\alpha_1 + \alpha_2 + 4\alpha_3 + \alpha_4) \\
	 & \quad + (3\alpha_1 + 4\alpha_2 + 2\alpha_3 + \alpha_4)D_{-1} + (1 + 2\alpha_1 + 4\alpha_2 + 3\alpha_3 + \alpha_4)D_{-1} D_{-2} \\
	& \quad\quad + (\alpha_1 + \alpha_2 + \alpha_3 + \alpha_4)D_{-1} D_2 D_{-2} \not\equiv 0 ~({\rm mod}~5).
\end{align*}
In the same manner, we can compute the formulas of $f^5(0)$ for the other five cases as follows.
\item[case] II: Consider the case II in Proposition \ref{prop:minimal_cond1}, which is that
\[
A_1\equiv A_2\equiv4, A_3\equiv3, \text{ and } A_4\equiv 0~({\rm mod}~5).
\]
Substitute $A_1 \equiv 4 + \alpha_1 5, A_2 \equiv 4 + \alpha_2 5 + 4, A_3 \equiv 3 + \alpha_3 5$, and $A_4 \equiv \alpha_4 5$ with $\alpha_i \in \Z_5$  and then we obtain
\begin{align*}
f^{5}(0) \equiv& 5\bigl((4 + 3\alpha_1 + 4\alpha_2 + 2\alpha_3 + \alpha_4) \\
&\quad + (4\alpha_1 + \alpha_2 + 4\alpha_3 + \alpha_4)D_{-2} + (2\alpha_1 + 4\alpha_2 + 3\alpha_3 + \alpha_4)D_{-1} D_{-2} \\
&\quad\quad + (2 + \alpha_1 + \alpha_2 + \alpha_3 + \alpha_4) D_{-1} D_2 D_{-2}\bigr) ~({\rm mod}~25).
\end{align*}
In order to have $f^5(0) \not\equiv 0 ~({\rm mod}~25)$, the terms derived from $f$ must satisfy
\begin{align*}
&(4 + 3\alpha_1 + 4\alpha_2 + 2\alpha_3 + \alpha_4) \\
&\quad + (4\alpha_1 + \alpha_2 + 4\alpha_3 + \alpha_4)D_{-2} + (2\alpha_1 + 4\alpha_2 + 3\alpha_3 + \alpha_4)D_{-1} D_{-2} \\
&\quad\quad + (2 + \alpha_1 + \alpha_2 + \alpha_3 + \alpha_4) D_{-1} D_2 D_{-2} \not\equiv 0 ~({\rm mod}~5).
\end{align*}
\item[case] III: Consider the case III in Proposition \ref{prop:minimal_cond1}, which is that
$$
A_1\equiv 1, A_2\equiv A_3\equiv3, \text{ and } A_4\equiv 0~({\rm mod}~5).
$$

Substitute $A_1 \equiv 1 + \alpha_1 5, A_2 \equiv 3 + \alpha_2 5, A_3 \equiv 3 + \alpha_3 5$, and $A_4 \equiv \alpha_4 5$ with $\alpha_i \in \Z_5$ and then we obtain
\begin{eqnarray*}
	f^{5}(0)&\equiv& 5D_{-1} ~({\rm mod}~25).
\end{eqnarray*}
and note that $a_1 D_1 D_2 D_{-2} D_{-1} \equiv 1 ~({\rm mod}~5)$ leads to $5D_{-1} \not\equiv 0~({\rm mod}~25)$, always.

\item[case] IV: Consider the case IV in Proposition \ref{prop:minimal_cond1}, which is that
$$
A_1\equiv 1, A_2\equiv4, A_3\equiv2, \text{ and } A_4\equiv 0~({\rm mod}~5).
$$

Substitute $A_1 \equiv 1 + \alpha_1 5, A_2 \equiv 4 + \alpha_2 5, A_3 \equiv 2 + \alpha_3 5$, and $A_4 \equiv \alpha_4 5$ with $\alpha_i \in \Z_5$  and then we obtain
\begin{align*}
	f^{5}(0)\equiv&5\bigl((4 + 2\alpha_1 + 4\alpha_2 + 3\alpha_3 + \alpha_4) \\
	&\quad + (4\alpha_1 + \alpha_2 + 4\alpha_3 + \alpha_4) D_2 + (3\alpha_1 +  4\alpha_2 + 2\alpha_3 + \alpha_4 ) D_{-1} D_2 \\
	&\quad\quad + (2 + \alpha_1 + \alpha_2  + \alpha_3 + \alpha_4) D_{-1} D_2 D_{-2} \bigr) ~({\rm mod}~25).
\end{align*}
In order to have $f^5(0) \not\equiv 0 ~({\rm mod}~25)$, the terms derived from $f$ must satisfy
\begin{align*}
&(4 + 2\alpha_1 + 4\alpha_2 + 3\alpha_3 + \alpha_4) \\
	&\quad + (4\alpha_1 + \alpha_2 + 4\alpha_3 + \alpha_4) D_2 + (3\alpha_1 +  4\alpha_2 + 2\alpha_3 + \alpha_4 ) D_{-1} D_2 \\
	&\quad\quad + (2 + \alpha_1 + \alpha_2  + \alpha_3 + \alpha_4) D_{-1} D_2 D_{-2} \not\equiv 0 ~({\rm mod}~5).
\end{align*}
\item[case] V: Consider the case V in Proposition \ref{prop:minimal_cond1}, which is that
$$
A_1\equiv 4, A_2\equiv A_3\equiv2, \text{ and } A_4\equiv 0~({\rm mod}~5).
$$

Substitute $A_1 \equiv 4 + \alpha_1 5, A_2 \equiv 2 + \alpha_2 5, A_3 \equiv 2 + \alpha_3 5$, and $A_4 \equiv \alpha_4 5$ with $\alpha_i \in \Z_5$  and then we obtain
\begin{align*}
	f^{5}(0)\equiv& 5\bigl((4 + 3\alpha_1 + 4\alpha_2 + 2\alpha_3 + \alpha_4) \\
	&\quad + (1 + 2\alpha_1 + 4\alpha_2 + 3\alpha_3 + \alpha_4) D_{-2} + (4 + 4\alpha_1 + \alpha_2 + 4\alpha_3 + \alpha_4) D_2 D_{-2} \\
	&\quad\quad + (2 + \alpha_1 + \alpha_2 + \alpha_3 + \alpha_4) D_{-1} D_2 D_{-2}\bigr)~({\rm mod}~25).
\end{align*}
In order to have $f^5(0) \not\equiv 0 ~({\rm mod}~25)$, the terms derived from $f$ must satisfy
\begin{align*}
&(4 + 3\alpha_1 + 4\alpha_2 + 2\alpha_3 + \alpha_4) \\
	&\quad + (1 + 2\alpha_1 + 4\alpha_2 + 3\alpha_3 + \alpha_4) D_{-2} + (4 + 4\alpha_1 + \alpha_2 + 4\alpha_3 + \alpha_4) D_2 D_{-2} \\
	&\quad\quad + (2 + \alpha_1 + \alpha_2 + \alpha_3 + \alpha_4) D_{-1} D_2 D_{-2}  \not\equiv 0 ~({\rm mod}~5).
\end{align*}
\item[case] VI: Consider the case VI in Proposition \ref{prop:minimal_cond1}, which is that
$$
A_1\equiv A_2\equiv 0, A_3\equiv3, \text{ and } A_4\equiv 0~({\rm mod}~5).
$$
Substitute $A_1 \equiv \alpha_1 5, A_2 \equiv \alpha_2 5, A_3 \equiv 3 + \alpha_3 5$, and $A_4 \equiv \alpha_4 5$ with $\alpha_i \in \Z_5$ and then we obtain
\begin{align*}
	f^{5}(0)\equiv& 5\bigl((4 + 2\alpha_1 + 4\alpha_2 + 3\alpha_3 + \alpha_4) \\
	&\quad + (1 + 3\alpha_1 + 4\alpha_2  + 2\alpha_3 + \alpha_4) D_2 +(4\alpha_1 + \alpha_2 + 4\alpha_3 + \alpha_4 )D_2 D_{-2} \\
	&\quad\quad + (1 + \alpha_1 + \alpha_2 + \alpha_3 + \alpha_4) D_{-1} D_2 D_{-2} \bigr)~({\rm mod}~25).
\end{align*}
In order to have $f^5(0) \not\equiv 0 ~({\rm mod}~25)$, the terms derived from $f$ must satisfy
\begin{align*}
&(4 + 2\alpha_1 + 4\alpha_2 + 3\alpha_3 + \alpha_4) \\
	&\quad + (1 + 3\alpha_1 + 4\alpha_2  + 2\alpha_3 + \alpha_4) D_2 +(4\alpha_1 + \alpha_2 + 4\alpha_3 + \alpha_4 )D_2 D_{-2} \\
	&\quad\quad + (1 + \alpha_1 + \alpha_2 + \alpha_3 + \alpha_4) D_{-1} D_2 D_{-2} \not\equiv 0~({\rm mod}~5).
\end{align*}
\end{description}
Therefore, the proof is completed.
\end{proof}

We present an example.
\begin{example}
The polynomial $f(x)=5x^5 + 10x^4 - 5x^2 - 4x + 1$ is a minimal map in $\Z_5$.
\end{example}
\begin{proof} The terms derived from $f$ are
\begin{align*}
	& A_1 =a_1+a_5 = 1, \quad\quad
    A_2 =a_2 = -5, &
    & A_3 =a_3 = 0, \quad\quad
     A_4 =a_4 = 10, \\
	& D_1=a_1 + 2a_2 - 2a_3 -a_4 = -24, &
	& D_{-1}=a_1 - 2a_2 - 2a_3 +a_4 = 16,\\
    & D_2=a_1 + 2a_2 2 - 2a_32^2 -a_4 2^4 = -104, &
    & D_{-2}=a_1 - 2a_2 2 - 2a_3 2^2 + a_4 2^3 = 96.
\end{align*}

Then we have that $A_1\equiv 1, \text{ and } A_2\equiv A_3\equiv A_4\equiv 0~({\rm mod}~5)$, which fits the case I in Theorem \ref{thm:minimal5}. Also, we obtain that
$$
\alpha_1 = 0, \alpha_2 = -1, \alpha_3 = 0, \alpha_4 = 2.
$$
We check that the second condition of the case I is satisfied:
\[ a_1 D_1 D_2 D_{-2} D_{-1} \equiv (-4)\cdot (-24)\cdot (-104)\cdot (96) \cdot 16 \equiv  1 ~({\rm mod}~5). \]
We check that the third condition of the case I is satisfied:
\begin{align*}
&(4\alpha_1 + \alpha_2 + 4\alpha_3 + \alpha_4) + (3\alpha_1 + 4\alpha_2 + 2\alpha_3 + \alpha_4)D_{-1} \\
&\quad + (1 + 2\alpha_1 + 4\alpha_2 + 3\alpha_3 + \alpha_4)D_{-1} D_{-2} + (\alpha_1 + \alpha_2 + \alpha_3 + \alpha_4)D_{-1} D_2 D_{-2} \\
& \quad\quad = -161311 \equiv 4 \not\equiv 0~(\rm{mod}~5).
\end{align*}

The polynomial $f$ satisfies the three conditions of the case I, therefore the dynamical system$(\Z_5, f)$ minimal. In fact, we can check directly by computing $(f^n (0)~\rm{mod}~25)$ for $n = 0,1,2,\cdots, 25$. From $0$ to $f^{25}(0)$, we list as follows:
$ 0 \rightarrow 1\rightarrow 7 \rightarrow 23 \rightarrow 14 \rightarrow 20 \rightarrow 21 \rightarrow 2 \rightarrow 18 \rightarrow 9 \rightarrow 15 \rightarrow 16 \rightarrow 22 \rightarrow 13 \rightarrow 4 \rightarrow 10 \rightarrow 11 \rightarrow 17 \rightarrow 8 \rightarrow 24 \rightarrow 5 \rightarrow 6 \rightarrow 12 \rightarrow 3 \rightarrow 19 \rightarrow 0.$

Since $f_{/2}$ is minimal, by Proposition \ref{prop:minimal}, $f$ is minimal.
\end{proof}

\section{Discussion}\label{sec:discu}

We provided a complete characterization of minimality of polynomial dynamical systems on the set of $p$-adic integers for $p=5$, using an extended method of the proof in \cite{DP}. We may expect that our method can apply to characterize the minimality of polynomial dynamical systems for a general prime $p$.

%



%
\end{document}